\documentclass[sn-mathphys-num]{sn-jnl}

\usepackage{graphicx}%
\usepackage{multirow}%
\usepackage{amsmath,amssymb,amsfonts}%
\usepackage{amsthm}%
\usepackage{mathrsfs}%
\usepackage[title]{appendix}%
\usepackage{xcolor}%
\usepackage{textcomp}%
\usepackage{manyfoot}%
\usepackage{booktabs}%
\usepackage{algorithm}%
\usepackage{algorithmicx}%
\usepackage{algpseudocode}%
\usepackage{listings}%

\def\dist{\ensuremath{\rho^1}}
\def\esp#1{\mathbb E\left[#1\right]}

\usepackage{float}
\newfloat{figure}{H}{lof}
\floatname{figure}{\figurename}

\usepackage[french,english]{babel}
\usepackage[T1]{fontenc}
\usepackage[utf8]{inputenc}
\usepackage{graphicx}
\usepackage{lmodern}
\usepackage{fourier}
\usepackage{graphicx}
\usepackage{comment}

\DeclareMathAlphabet{\eufrak}{U}{}{}{}
\SetMathAlphabet\eufrak{normal}{U}{euf}{m}{n}
\SetMathAlphabet\eufrak{bold}{U}{euf}{b}{n}

\def\E{\mathbb{E}}
\def\P{\mathbf{P}}

\def\1{\textbf{1}}
\def\ind#1{\mathbf{1}_{\{#1\}}}
\newcommand{\indic}[1]{\ensuremath{\mathbf 1_{#1}}}

\def\Lip{\operatorname{Lip}}

\renewcommand{\[}{\begin{eqnarray*}}
\renewcommand{\]}{\end{eqnarray*}}

\newcommand{\bex}{\begin{ex} \rm }
\newcommand{\eex}{\end{ex}}

\def\R{\mathbb R}
\def\N{\mathbb N}
\def\E{\mathbb E}

\def\esp#1{\mathbf E\left[#1\right]}

\def\cal#1{\mathcal{#1}}

\definecolor{ying}{rgb}{0.8, 0.0, 0.04}

\def\dif{\text{ d}}
\DeclareSymbolFontAlphabet{\mathrsfs}{rsfs}

\theoremstyle{thmstyleone}%
\newtheorem{theorem}{Theorem}
\newtheorem{prop}[theorem]{Proposition}%
\newtheorem{lemma}[theorem]{Lemma}%
\newtheorem{corollary}[theorem]{Corollary}%

\theoremstyle{thmstyletwo}%
\newtheorem{remark}{Remark}%

\theoremstyle{thmstylethree}%

\begin{document}

\title[Rate of Convergence in the Functional Central Limit Theorem for Stable Processes]{Rate of Convergence in the Functional Central Limit Theorem for Stable Processes}


\author*[1,2]{\fnm{Laure} \sur{Coutin}}\email{laure.coutin@math.univ-toulouse.fr}

\author[2,3]{\fnm{Laurent} \sur{Decreusefond}}\email{laurent.decreusefond@mines-telecom.fr}
\equalcont{These authors contributed equally to this work.}

\author[1,2]{\fnm{Lorick} \sur{Huang}}\email{lhuang@insa-toulouse.fr}
\equalcont{These authors contributed equally to this work.}

\affil*[1]{\orgdiv{Institut Mathématiques de Toulouse}, \orgname{Universit\'e Paul Sabatier 118}, \orgaddress{\street{route de Narbonne}, \city{Toulouse Cedex 4}, \postcode{31062 }, \country{France}}}

\affil[2]{\orgdiv{Institut Polytechnique de Paris}, \orgname{Telecom Paris}, \orgaddress{\street{19, place M. Perey}, \city{Palaiseau}, \postcode{91120},  \country{France}}}

\affil[3]{\orgdiv{Institut Mathématiques de Toulouse}, \orgname{INSA de Toulouse}, \orgaddress{\street{135 avenue de Rangueil}, \city{Toulouse Cedex 4}, \postcode{31077}, \country{France}}}


\abstract{In this article, we quantify the functional convergence of the rescaled random walk with heavy tails to a stable process.
This generalizes the  Generalized Central Limit Theorem  for stable random variables in
finite dimension. We show that provided we have a control between  the random
walk or the limiting stable process and their respective affine interpolation, we can
lift the rate of convergence obtained for multivariate distributions to a rate
of convergence in some functional spaces.
}




\maketitle

\section{Introduction}

\subsection{Motivations}
\label{sec:motivations}

The Stein's method is one way to bound the Wasserstein-1  (\dist{}
for short) distance between two
probability measures on a metric space $(E,d)$, defined by
\begin{equation}
\label{eq_TCL_short_core:1} \dist_{E}(\mu,\nu)=\sup_{F\in\Lip_{1}(E)}\left(\int_{E}F\dif \mu-\int_{E}F\dif \nu\right)
\end{equation}
where
\begin{equation*}
\Lip_{1}(E)=\Bigl\{F\, :\, E\to \R,\ |F(x)-F(y)|\le d(x,y),\ \forall x,y\in E\Bigr\}.
\end{equation*}
It is known \cite{Villani2003} that convergence in the \dist{}
 distance implies
the convergence in distribution and the convergence of first moments. The
Stein's method initiated in the seventies by C. Stein, was originally built to
assess rate of $\dist_{\R}$ convergence towards the Gaussian distribution. It was
soon extended to the Poisson limit by L. Chen (see
\cite{10.1214/20-PS354} and references therein). Since then,
the majority of the tremendous number of papers dealing with this approach has
been concentrated on Gaussian limits and in a tinier percentage to Poisson
limits. It is only very recently that some attention were devoted to some other
limiting regimes, like convergence to a stable distribution
\cite{Arras2015a,arras:houdre,Upadhye2022,Barman2022,chen:nourdin:xu:2018,Chen2021a}. There is here a
dramatic difference between the cases $\alpha>1$ and $\alpha\le 1$, since in the
latter situation the distribution has no longer a finite first moment. We refer to monographs such as \cite{zolotarev, feller2, meerschaert:sikorskii:book, bertoin, samo:taqqu} for an overview of non-Gaussian CLT.
 In view of the preceding remark, this means that we cannot expect a convergence in the
sense of the Wasserstein-1 distance but we have to restrict the set of test
functions taken in the supremum  of \eqref{eq_TCL_short_core:1} as in
\cite{Chen2021a}. We will not deal with this problem here but it is the object a
paper in preparation \cite{CDH24}. We thus assume throughout this
work that $\alpha$ belongs to $(1,2)$.
Another extension of the Stein techniques is to consider functional convergence
theorems, i.e. CLT-like theorems in spaces of functions. The first paper to handle
such a situation was \cite{barbour:1990}. More recently, in a
series of paper
\cite{besancon:hal-03283778,coutin:decreusefond:2020,Coutin:2019aa,Coutin2014,coutin_steins_2013},
the rate of convergence of some classical theorems like the Donsker Theorem or the Brownian approximation of
a normalized Poisson process were established. There are two remarks which can
be made here. First, the rate of convergence depends on the functional space
into which we are envisioning the processes under study. For instance, the
sample-paths of a Brownian motion can be seen as continuous functions or as
Hölder continuous functions: The rate of convergence of the random walk towards
the Brownian motion has been shown to be $n^{-1/6}\log n$ (respectively $n^{-1/6+\gamma/3}$) in the space of
continuous functions (respectively the space of $\gamma$-Hölder continuous
functions for $\gamma<1/2$). The second remark is that we can get rid of the
calculations in infinite dimension which appeared in \cite{coutin_steins_2013}
by considering affine interpolations of the different processes involved and
then reduce the computations to estimates in some finite dimensional spaces.

The main contribution of this paper is to show how we can
convert a rate of convergence for multivariate distributions into a
rate of convergence in some Banach spaces provided we have a control on the error made by replacing some
processes by their affine interpolation. In passing, we greatly simplify  one of our former proof developed for the Donsker theorem in \cite{coutin:decreusefond:2020}. We treat here the case of a random walk
with increments in the domain of attraction of a stable distribution with finite
mean but infinite variance. As a key element  of our proofs, we derive a moment bound
for sums of stable random variables which seems to be new and interesting by itself.

\subsection{Main result}
\label{sec:main-result}

A random variable $Y$ is said to be in the domain of attraction of a stable
distribution if for a collection of random variables $Y_1, \dots, Y_n$ independent and with the same distribution as $Y$, there exists two sequences $(a_n)_{n \ge 0}$ and
$(b_n)_{n \ge 0}$ such that
$$
a_n(Y_1+ \cdots + Y_n )-b_n \xrightarrow[n\to \infty]{\text{law}} S,
$$
where $S$ is an $\alpha$-stable random variable, for $\alpha \in (0,2]$. In the
literature, the appartenance to the domain of attraction has been linked to the
tail of the distribution, see e.g. \cite{bertoin} or
\cite{meerschaert:sikorskii:book} for the review of the literature, and complete
statement of the theorems.

One of the first paper to use Stein's method to derive rates of convergence for the stable central limit is   \cite{chen:nourdin:xu:2018}. In that paper, the authors introduce the normal domain of attraction of an $\alpha$-stable distribution:  random variables whose cumulative distribution functions is of the form
\begin{equation}\label{normal:domain:attraction}
1-F_Y(t) = \P(Y \ge t) = \frac{A+ \varepsilon(t)}{t^\alpha} \mbox{ and } F_Y(-t) = \frac{A+ \varepsilon(-t)}{(-t)^\alpha},
\end{equation}
whenever $t \ge 1$, and where $\varepsilon$ is a function with  a specified
decay at infinity.
For $(Y_{n})_{n\ge 1}$ a sequence of IID random variables of this sort, define
\begin{equation*}
  S_{n}=\frac{1}{n^{1/\alpha}}  \sum_{i=1}^{n}Y_{i}.
\end{equation*}
It is
shown in  \cite{chen:nourdin:xu:2018} that:
\begin{equation*}
  \dist_{\R}(S_{n},S)\le \frac{c}{n^{1-2/\alpha}}\cdotp
\end{equation*}
Building on that paper, and their subsequent work, we aim at establishing  the rate of  convergence of the
continuous time analog of $S_{n}$, namely the random walk $(X_n(t))_{t\ge 0}$
defined by
\begin{equation}
\label{eq_TCL_short_core:2}X_{n}(t)=\frac{1}{2^{n/\alpha}}\sum_{i=1}^{[2^nt]}Y_{i},
\end{equation}
to an $\alpha$ stable process $(S_t)_{t \ge 0}$ in a space of functions. As in
 some of our previous works, given  $p\ge 1$ and
$\eta\in [0,1]$,  we introduce the  fractional Sobolev space
$$
W_{\eta,p}= \left\{ f \in L^p([0,1],\R^d); \ \ \int_0^1 \int_0^1 \frac{|f(s)-f(t)|^p}{|s-t|^{1+\eta p}} \dif s \dif t <+\infty \right\},
$$
equipped with the norm
$$
\Vert f\Vert_{W_{\eta,p}}^p= \int_0^1 |f(t)|^p \dif t + \int_0^1 \int_0^1 \frac{|f(s)-f(t)|^p}{|s-t|^{1+\eta p}} \dif s \dif t.$$
Note that for $\eta-1/p>0$,
$W_{\eta,p}$ is embedded in the space $(\eta-1/p)$-Hölder continuous functions
whereas for $\eta-1/p<0$, $W_{\eta,p}$ can be embedded into the space of
$p(1-\eta p)^{-1}$ integrable functions over $[0,1]$. Since $(X_n(t))_{t\ge 0}$
and $(S_t)_{t \ge 0}$ induce probability distributions on some of  these functional spaces
$W_{\eta,p}$ (for $\eta<\alpha^{-1}<1\le p<\alpha$), their distance can be
quantified by the
Wasserstein-1 distance on any of these  spaces defined  as
\begin{equation*}
  \dist_{W_{\eta,p}}\Bigl(\text{law}(X_{n}),\ \text{law}(S)\Bigr)=\sup_{F\in \Lip_{1}(W_{\eta,p})}\Bigl(\esp{F(X_{n})}-\esp{F(S)}\Bigr).
\end{equation*}


Our main result is an upper-bound  of this rate of convergence:
\begin{theorem}
  Let $(Y_{i})_{i\ge 1}$ be a sequence of IID random variables of cdf $F_{Y}$, in
  the normal domain of attraction of an $\alpha$-stable distribution with
  $\alpha\in (1,2)$. 
  We assume that their exists $A>0,$ $\gamma>0$ and a function $\varepsilon$ such that \begin{align*}
      &{\mathbb P}[Y_1 > t]= \frac{A}{2} \frac{1}{t^{\alpha}}{\mathbf 1}_{[1,+\infty[}(t) + (1-A) \frac{\varepsilon(t)}{t^{\alpha}}~~\mbox{~~for~~} ~~t> 0,\\
      &{\mathbb P}[Y_1 \leq  t]= \frac{\eta}{2} \frac{1}{|\min(t,1)|^{\alpha}} + (1-A) \frac{\varepsilon(t)}{t^{\alpha}}~~\mbox{~~for~~}~~ t\leq 0,
  \end{align*}
  where $\varepsilon$ is a bounded function on $[-1,1]$ such that $\sup_{t}|t^{\gamma}\varepsilon(t)| <+\infty. $
  Let $X_{n}$ be defined as in \eqref{eq_TCL_short_core:2}.
  For $(\eta,p)\in (0,1/\alpha]\times [1,\alpha)$, there exist
  $\upsilon>0$ and $c>0$ both depending on $(\eta,p,\alpha,\gamma)$ such that we have
  \begin{equation*}
\dist_{W_{\eta,p}}(X_n ,S) \le c\, 2^{-n\upsilon},
  \end{equation*}
  where
  \begin{equation}
    \label{eq_TCL_short_core:4}
    \upsilon = 
    \left(\frac{1}{\alpha}-\eta\right)p 
\frac{\min\left(\frac{2}{\alpha}-1, \frac{\gamma}{\alpha} \right)}{1+\left(\frac{1}{\alpha}- \eta\right)p}.
    >0.
  \end{equation}  
\end{theorem}
The strategy we use here is inspired by \cite{coutin:decreusefond:2020}. First, in Section \ref{sec:Rand:Walk}, we consider a projection of $(X_n(t))_{t\ge 0}$ and $(S_t)_{t \ge 0}$ on a certain finite dimensional space.
This reduction to the finite dimension allows us to rely on  \cite{chen:nourdin:xu:2018,chen:nourdin:xu:yang:2019,chen:nourdin:xu:yang:zhang:2022} to control the distance of the two projected processes.
Then, the control the sample-paths distances between
$(X_n(t))_{t\ge 0}$ and $(S_t)_{t \ge 0}$  and their respective projections is obtained via Lemma \ref{ctrl:norme:holder} of Section \ref{section:ctrl:norme:holder}. 
Finally, we summarize our results and state the final control in Section \ref{final:derivation}. 
As a key element  of our approach, we establish a moment bound in Section \ref{section:moment:cond:proof} that, up to our knowledge, was not known, even for Pareto distributions. We prove that when $Y_1, \dots Y_n$ are in the normal domain of attraction of a stable distribution, for all $p <\alpha$,
$$
\sup_{n}\esp{  \left| \frac{Y_1+ \cdots + Y_n}{n^{1/\alpha}} \right|^p }<+\infty.
$$

\section{The Interpolated Random Walk, and Reduction to Finite Dimension}
\label{sec:Rand:Walk}
\subsection{Notations and the Interpolated Random Walk}

For the sake of simplicity, we choose the dyadic partition to define the affine
interpolations of the random walk. The general situation, for which we have non
nested intervals and thus side effects, can be handled as in
\cite{coutin:decreusefond:2020}. We consider the sequence of functions:
$$
h_n^i(t) = \sqrt{2^{n}} \int_0^t\indic{[\frac{i}{2^{n}},\, \frac{i+1}{2^{n}} )} (s)
\dif s.
$$
Notice that this sequence actually forms an orthonormal family with respect to
the inner product
$$
\langle f,g \rangle = \int_0^1f'(s)g'(s) \dif s.
$$
We then form the interpolated random walk:
\begin{align*}
  X_n(t)
  &= \frac{\sqrt{2^n}}{2^{n/\alpha}} \sum_{i=1}^{2^n} Y_i h_n^{i-1} (t)\\
  &= \frac{1}{2^{n/\alpha}} \sum_{i=1}^{\left \lfloor{2^nt}\right \rfloor -1} Y_i +  \frac{Y_{\left \lfloor{2^nt}\right \rfloor }}{2^{n/\alpha}} \Big( 2^nt - \left \lfloor{2^nt}\right \rfloor \Big).  
\end{align*}
To obtain a convergence rate, we introduce for $m\le n$ (to be determined later)
the projection operator on $\operatorname{span}\{h_m^j, j=0,\cdots,2^m-1\}$:
$$
\pi_m(f) = \sum_{j=0}^{2^m-1} \langle f, h_m^j\rangle\ h_m^j.
$$
The introduction of this projection operation is the tool that allows us to use results available in finite dimension.
We split the difference as follows:
$$
X_n - S = \Bigl(X_n - \pi_m(X_n)\Bigr) + \Bigl(\pi_m(X_n) - \pi_m(S)\Bigr)+\Bigl(\pi_m(S)-S\Bigr).
$$
The first and third terms are of the same nature, they represent the gap between
a process and its projection.
In Section \ref{section:ctrl:norme:holder}, we prove the  Lemma
\ref{ctrl:holder:process} that deduces a control of theses distances from H\"older sample-path regularity. For the second term above, we will see that this term is controlled
using the finite-dimensional distribution convergence, which is estimated using
Stein's method in finite dimension. For that term, we use the control given by  \cite{chen:nourdin:xu:2018}, but we also provide new approaches below.

\subsection{Reduction to finite dimension}

In this section, we show how the the projection operator allows one to reduce the analysis on finite dimensional spaces, from where we can cite
existing work in the literature.
This adapts arguments developed initially in~\cite{coutin:decreusefond:2020}.
\begin{lemma}
  The projection of the partial sums write:
$$\pi_m(X_n)(t) = \sum_{j=1}^{2^m} \langle X_n, h_m^j\rangle h_m^j(t)= \sum_{j=1}^{2^m}  \frac{\sqrt{2^m}}{2^{\frac{n-m}{\alpha}}}
\left(\frac{1}{2^{\frac{m}{\alpha}}} \sum_{i=j2^{n-m}}^{(j+1)2^{j+1}} Y_i \right) h_m^{j-1}(t).
$$\end{lemma}

\begin{proof}

  Recall that
  $X_n(t) = \frac{\sqrt{2^n}}{2^{n/\alpha}}\sum_{i=1}^{2^n} Y_i h_n^{i-1}(t).$
  Taking the projection leads to
$$\pi_m(X_n)(t) =\frac{\sqrt{2^n}}{2^{n/\alpha}} \sum_{j=1}^{2^m}\sum_{i=1}^{2^n}  \langle h_n^{i-1},h_m^{j-1}\rangle h_m^{j-1}(t) Y_i .
$$Now, the scalar product $\langle h_n^i,h_m^j\rangle$ is zero except when the intervals $\left[ \frac{i}{2^n}, \frac{i+1}{2^n}\right]$ and $\left[ \frac{j}{2^n}, \frac{j+1}{2^n}\right]$ are nested. In that case, the scalar product yields the length of the resulting interval, hence $\langle h_n^i,h_m^j\rangle =\frac{\sqrt{2^n} \sqrt{2^m}}{2^{n}}.$
Thus, the projection becomes
$$\pi_m(X_n)(t) =\frac{\sqrt{2^n}}{2^{n/\alpha}} \sum_{j=1}^{2^m}\sum_{i=j2^{n-m}}^{(j+1)2^{n-m}-1}  \frac{\sqrt{2^n} \sqrt{2^m}}{2^{n}} h_m^{j-1}(t) Y_i .
$$We simplify the terms and normalise the inner sum by $2^{\frac{n-m}{\alpha}}$ in the above expression to get:
$$\pi_m(X_n)(t)  = \frac{\sqrt{2^m}}{2^{\frac{m}{\alpha}}}
\sum_{j=1}^{2^m}\left( \frac{1}{2^{\frac{n-m}\alpha}}\sum_{i=j2^{n-m}}^{(j+1)2^{n-m}-1} Y_i \right)h_m^j(t).
$$
The proof is thus complete.
\end{proof}

\begin{remark}
    In the proof above, we see that considering the dyadic partition simplifies the discussion, since the intervals  $\left[ \frac{i}{2^n}, \frac{i+1}{2^n}\right]$ and $\left[ \frac{j}{2^n}, \frac{j+1}{2^n}\right]$ are nested. In the case of an arbitrary partition, we would have boundary terms that needs to be handled. We refer the reader to \cite{coutin:decreusefond:2020}, where such a discussion is performed.
\end{remark}

Meanwhile, we have that the  projection $\pi_m(S)$ is:
$$\pi_m(S)(t)= \sum_{j=1}^{2^m}\langle S, h_m^{j-1}\rangle h_m^{j-1}(t).$$ 
This scalar product can be computed with an integration by parts:
$$
\langle S, h_m^j\rangle =  \sqrt{2^m} \int_{[0,1]} {S}'(t)\ \indic{[\frac{j}{2^m},\frac{j+1}{2^m})}(t) \dif t =\sqrt{2^m}  \left(S\left(\frac{j+1}{2^m}\right) - S\left(\frac{j}{2^m} \right) \right).
$$
Thus, we have:
$$\pi_m(S)(t) = \sum_{j=0}^{2^m-1} h_m^j(t)\sqrt{2^m}  \left(S\left(\frac{j+1}{2^m}\right) - S\left(\frac{j}{2^m} \right) \right)
$$and each term $S\left(\frac{j+1}{2^m}\right) - S\left(\frac{j}{2^m} \right) \overset{(d)}{=} \frac{1}{2^{m/\alpha}}S_1.$
Hence, for all $t\in[0,1]$, we have established
that 
\begin{multline*}
  F\Bigl( \pi_m(X_n)(t)\Bigr) - F\Bigl( \pi_m(S)(t)\Bigr)=\\
  F\left(\frac{\sqrt{2^m}}{2^{\frac{m}{\alpha}}} \sum_{j=0}^{2^m-1}\left( \frac{1}{2^{\frac{n-m}\alpha}}\sum_{i=j2^{n-m}}^{(j+1)2^{n-m}-1} Y_i \right)h_m^j(t)\right) - F \left(\sqrt{2^m} \sum_{j=0}^{2^m-1} \left(S\left(\frac{j+1}{2^m}\right) - S\left(\frac{j}{2^m} \right) \right) h_m^j(t) \right) .
\end{multline*}
Let
$$\frac{1}{2^{\frac{n-m}\alpha}}\sum_{i=j2^{n-m}}^{(j+1)2^{n-m}-1}  Y_i =U_{m,j}^{n} \mbox{, and } S\left(\frac{j+1}{2^m}\right) - S\left(\frac{j}{2^m} \right) =S_{m,j},$$ we have obtained:
\begin{equation}\label{eq_chainon_manquant:1}
  \esp{ F\Big( \pi_m(X_n)\Bigr)}\ - \esp{F\Big( \pi_m(S)\Big)}=\esp{ F\left(\sum_{j=0}^{2^{m}-1}U_{m,j}^{n} h_{j}^{m}\right)}- \esp{ F\left(\sum_{j=0}^{2^{m}-1}S_{m,j} h_{j}^{m}\right)},
\end{equation}
where $U_{m,j}^{n}$ and $S_{m,j}$ are IID. Using the equivalent
characterizations of the Wasserstein-1 distance, we know that there exists a
coupling (not necessarily unique) between the random variables $U_{m,j}^{n}$ and $S_{m,j}$ that
realizes the distance $\dist_{\R^{d}}(U_{m,j}^{n},\, S_{m,j})$
for any $j\in \{0,\cdots,2^{m}-1\}$.
Let $(\hat U_{m,j}^{n},\hat S_{m,j})$ be independent random  variables  such that for any $j \in \{0,\cdots,2^{m-1}-1\}$
\begin{equation*}
\dist_{W}(U_{m,j}^{n},\, S_{m,j})=\esp{\left|\hat U_{m,j}^{n}-\hat S_{m,j}\right|}.
\end{equation*}
The difference \eqref{eq_chainon_manquant:1} now gives:
\begin{align*}
  \dist_{W_{\eta,p}}\Big(\pi_m(X_n),\pi_m(S)\Big) &= \sup_{F \in \Lip_{1}(W_{\eta,p})}\esp{ F\Big( \pi_m(X_n)\Big)} - \esp{F\Big( \pi_m(S)\Big)} \\
                                              &= \sup_{F \in \Lip_{1}(W_{\eta,p})}\esp{ F\left(\sum_{j=1}^{2^{m}}\hat U_{m,j}^{n} h_{j}^{m}\right)}- \esp{ F\left(\sum_{j=1}^{2^{m}}\hat S_{m,j} h_{j}^{m}\right)}\\
                                              &\le \sum_{j=0}^{2^{m}-1}   \dist_{\R}(U_{m,j}^{n},\ S_{m,j}) \|h_{j}^{m}\|_{W_{\eta,p}}\\
  &\le c\,  2^{m}2^{-m(1/2-\eta)}  \dist_{\R}(U_{m,1}^{n},\ S_{m,1})\\
  &\le c 2^{m(1/2+\eta)} \dist_{\R}(U_{m,1}^{n},\ S_{m,1}).
\end{align*}

The last part $\dist_{\R^{d}}(U_{m,1}^{n},\ S_{m,1})$ is now bounded using the estimate of \cite{chen:nourdin:xu:2018,chen:nourdin:xu:yang:2019}:
$$\dist_{\R}(U_{m,1}^{n},\ S_{m,1})= \dist_{\R^{d}}\left( \frac{1}{2^{\frac{n-m}{\alpha}}} \sum_{i=j2^{n-m}}^{(j+1)2^{n-m}}  \frac{1}{2^{m/\alpha}}Y_i , j=1\dots,2^m, \textbf{S}\right) \le C_{\alpha,m} 2^{n(1-\frac{2}{\alpha})}.
$$
We explain how to use \cite{chen:nourdin:xu:2018,chen:nourdin:xu:yang:2019} in the next paragraph.

\begin{remark}
    It should be noted that the characterization of the Wasserstein-1 distance as a coupling is new in this setting. It simplifies greatly the proof of the analog  result in \cite[Theorem 4.7]{coutin:decreusefond:2020}.
\end{remark}
\subsection{Control in Finite Dimension}

The purpose of this section is to explain how to derive the functional
convergence from existing results in the literature. Many authors have
investigated the stable central limit using Stein's method, starting from
Barbour \cite{barbour:1990}, and more recently, a series of paper
\cite{chen:nourdin:xu:2018}, \cite{chen:nourdin:xu:yang:2019},
\cite{chen:nourdin:xu:yang:zhang:2022}. We will rely on their result in the
multidimensional setting.

The rate of convergence for heavy-tailed random walk
converging to stable distributions is obtained in \cite{chen:nourdin:xu:2018}. In a
subsequent work \cite{chen:nourdin:xu:yang:2019}, this result is
extended from one dimension  to the multivariate case. We copy their
result here keeping the notations used in \cite{chen:nourdin:xu:yang:2019}.
\begin{theorem}
  \label{thm_TCL:cnz}
  Set
  \begin{equation*}
    S_{n}=\big(\zeta_{n,1}-\esp{\zeta_{n,1}}\big)+\big(\zeta_{n,2}-\esp{\zeta_{n,2}}\big)+\cdots+\big(\zeta_{n,n}-\esp{\zeta_{n,n}}\big);
  \end{equation*}
  Let $\mu$ be an $\alpha-$stable distribution $\alpha\in(1,2)$ with L\'evy
  measure
$$ \nu(d\theta)=a g(\theta)d\theta+b \sum_{i=1}^d (\sigma_i \delta_{e_i} + \sigma_i'\delta_{-e_i} ) + c \nu_\gamma(d\theta).
$$Then we have:
$$\dist_{W}\big(\mathcal{L}(S_{n}),\mu\big)\leq C_{\alpha,d}\left(n^{-\frac{1}{\alpha}}\sum_{i=1}^{n}\esp{\left|\zeta_{n,i}\right|^{2-\alpha}}+n^{-\frac{1}{\alpha}}\sum_{i=1}^{n}\esp{|\zeta_{n,i} |}^{2}+\esp{\sum_{i=1}^{n}|\mathcal{R}_{n,i}|} \right),
$$where the remainder $\mathcal{R}_{n,i}$ is related to the distribution function of the $\zeta_{n,i}$ and the L\'evy measure of the stable distribution.
%
\end{theorem}

They proceed to give an estimation for the remainder in the case where
$\zeta_{n,i}$ are symmetrized Pareto. The following upper bound can be found in
\cite{chen:nourdin:xu:yang:2019} section 5:
\begin{corollary}
  Let $\tau_i=\varepsilon_i \xi_i$, where $\varepsilon_i$ is a random unit
  vector such that $\P(\varepsilon_i \in A) = \nu(A)$, and $\xi$ is a Pareto
  random variable. Define
$$S_n= \left(\frac{\alpha}{\dist_\alpha}\right)^{1/\alpha} \frac{1}{n^{1/\alpha}} \left( \Big(\tau_1 - \esp{\tau_{1}} \Big) + \cdots +\Big(\tau_n - \esp{\tau_{n}}\Big) \right).
$$Then,
$$\dist_W(\mathcal{L}(S_n),\mu) \le C_{\alpha,d} n^{\frac{\alpha -2}{\alpha}}.
$$

\end{corollary}

Recall we need a convergence rate for each $j=1, \dots, 2^m$:
$$\frac{1}{2^{\frac{n-m}{\alpha}}} \sum_{i=j2^{n-m}}^{(j+1)2^{n-m}}  \frac{1}{2^{m/\alpha}}Y_i \Rightarrow S\left(\frac{j+1}{2^m}\right) - S\left(\frac{j}{2^m} \right).
$$
Since each of the $2^m$ component are independent, the vector:
$$\left(S\left(\frac{1}{2^m}\right) , S\left(\frac{2}{2^m}\right) - S\left(\frac{1}{2^m} \right)  , \dots,  S\left(1\right) - S\left(\frac{2^m-1}{2^m} \right) \right) \overset{(d)}{=} \textbf{S},
$$has a stable distribution, whose L\'evy measure is
$$
\sum_{j=1}^{2^m} \frac12 \left( \delta_{e_j} + \delta_{-e_j} \right),
$$

Now, to use Theorem \ref{thm_TCL:cnz}, we need to compute each term in the right hand side, and $\zeta_{n,i}$ in terms of our to our $Y_i$.
Fortunately, in \cite{chen:nourdin:xu:yang:2019}, Paragraph 5.2, Example 2 (refer to the arXiv version 1), is a section devoted to the calculation we need. Using the same notations as in \cite{chen:nourdin:xu:yang:2019}, for a distribution function of the form 
$$
F_\xi(dr d\theta) = \frac{A}{r^{\alpha+1}} dr \nu(d\theta) + \frac{B(r\theta)}{r^{\beta+1}} dr d\theta,
$$
the convergence rate towards $\pi$, the multidimensional stable distribution, is:
$$
\rho_W^1({\cal L}(S_n), \pi) \le C_{\alpha,d,A,B} \Big( n^{\frac{\alpha-2}{\alpha}} +n^{\frac{\alpha-\beta}{\alpha}} \Big),
$$
where $d$ is the dimension, that is in our case, $2^m$. 
Recall that in our setting, we consider $Y_i$ in the normal domain of attraction (see Equation \eqref{normal:domain:attraction}), we now actually have to prescribe a decay rate for $\varepsilon$. 
From now on, assume that 
$$
\varepsilon(x) \le \frac{K_\varepsilon}{|x|^\gamma}.
$$
Note that this is the decay rate prescribed in \cite{chen:nourdin:xu:2018}, this relate to the Example discussed in \cite{chen:nourdin:xu:yang:2019} by taking $\gamma = \beta - \alpha \ge 0$.
Hence, in our case, the convergence rate is
$$
\rho_W^1({\cal L}(S_n), \pi) \le C_{\alpha,d,A,B} \Big( 2^{\frac{\alpha-2}{\alpha}n} +2^{\frac{-\gamma}{\alpha}n} \Big),
$$
where $\gamma$ relate to the choice of $\varepsilon$.
At this point, we need to be very careful with the constant $C_{\alpha,d,A,B}$, as a
compromise between $n$ and $m$ is to be chosen. In \cite{chen:nourdin:xu:yang:2019}, this
constant comes from the density estimation for the stable distribution. Since we
deal with symmetric stable distribution (that are unimodal with a mode at zero),
it is enough to upper bound the density at $x=0$ to get a good estimate on that
constant. Let us denote $\textbf{q}(\textbf{x})d\textbf{x}$ the density of
$\textbf{S}$, we have:
\begin{multline*}
\textbf{q}(0) = \frac{1}{(2\pi)^{2^m}}\int_{\R^{2^m}} \esp{ e^{ \langle\xi , \textbf{S} \rangle}} \dif\xi\\
                =\frac{1}{(2\pi)^{2^m}}\int_{\R^{2^m}} \exp \left(
                    \left(e^{i\langle \xi, \theta} -1 - i \langle \xi, \theta \rangle \textbf{1}_{\{|\theta|\le 1 \}}\right) \right.\\
                 \left. \times\Big(a g(\theta)\dif\theta+b \sum_{i=1}^d (\sigma_i \delta_{e_i} + \sigma_i'\delta_{-e_i} ) + c \nu_\gamma(\dif\theta)\Big)
                   \right) \dif \xi.
\end{multline*}
In general, one would proceed by change of variable looking for Gamma functions
to compute this integral. We point out that if the L\'evy measure had had a
density with respect to Lebesgue measure, the volume of the sphere $S^{2^m-1}$
would have come into consideration. However, because stable distribution
$\textbf{S}$ we consider has independent increments, this constant becomes
linear in the dimension $2^m$:

We therefore obtain the rate in Wasserstein distance:
$$\dist_{\R^{d}}\left( \frac{1}{2^{\frac{n-m}{\alpha}}} \sum_{i=j2^{n-m}}^{(j+1)2^{n-m}}  \frac{1}{2^{m/\alpha}}Y_i , j=1\dots,2^m, \textbf{S}\right) \le C_{\alpha,A,\varepsilon}2^m \Big( 2^{n(1-\frac{2}{\alpha})} + 2^{-n\frac{\gamma}{\alpha}}\Big).
$$

\section{ Approximation and Final Estimate}
\label{section:ctrl:norme:holder}
In what follows, the constants may change from line to line.
\subsection{Approximation Lemma}

Let $\tau_m=(t_k^m)$ where $t^n_k=k2^{-m},$ $k=0,...,2^m$ be the dyadic partition of $[0,1].$
Let $F$ be a stochastic process and $\pi^m(F)$ be its linear interpolation along $\tau^m,$ that is
\begin{align*}
    \pi_m(F)_t = F_{t^m_k}+ [t-t_m^k]2^m[F_{t^m_{k+1}} -F_{t_k^m}],~~t \in [t_k^m,t_{k+1}^m],~~k=0,...,2^m -1.
\end{align*}

\begin{lemma}\label{ctrl:norme:holder} Let  $1<p<\alpha <2,$ let $F$ be a
  stochastic process such that there exists $C>0$ such that
\begin{equation}\label{ctrl:holder:process}
  \esp{|F_t -F_s|^p}\leq C\, |t-s|^{p/\alpha},\ \forall \,s,t \in [0,1]^2.
\end{equation}
Then for $\eta <1/\alpha$, there exists a constant $C>0$ (depending on
$\alpha,p,\eta$) such
that, for all $m\ge 1$:
\begin{align*}
I^m&= \int \int _{[0,1]^2} \frac{\esp{|\pi_m(F)_t -\pi_m(F)_s -F_t +F_s|^p}}{|t-s|^{1+p\eta}}\dif s\dif t \leq C\ 2^{-m(1/\alpha -\eta)p}.\\
J^m&= \int_0^1 \esp{|\pi_m(F)_t-F_t|^p}\dif t \leq C\  2^{-m\frac{p}{\alpha}}.
\end{align*}
\end{lemma}
\begin{remark}
   The stable process satisfies the assumption of Lemma
   \ref{ctrl:holder:process}, since for any $0\le s<t\le 1$, we have
   \begin{equation*}
S(t)-S(s)\stackrel{\text{d}}{=}(t-s)^{1/\alpha}S(1)
\end{equation*}
and $S(1)\in L^{p}$ for any $p<\alpha$.
\end{remark}

\begin{remark}
 Since $p/\alpha<1$, \eqref{ctrl:holder:process} does not entail that $F$ has
continuous sample-paths.
Moreover, note that we have $I^m+J^m = \esp{\Vert F-\pi_m(F)\Vert^p_{W_{\eta,p}}}$. Hence,  this Lemma
gives a control of the norm, in $W_{\eta,p}$, of the difference
between a process and its affine interpolation provided that we have a bound on
the moments of its increments.
\end{remark}

\begin{proof}
Without loss of generality, let $s <t$.
First, let us assume that $s$ and $t$ are elements of $\tau_m$, that is $s=k2^{-m}, t=l2^{-m}$.  Then $F_t-F_s= \pi_m(F)_t-\pi_m(F)_s$, in other words,  the projection agrees with the process on the points on the partition.

Let $T^m_k=[t_k^m,t^m_{k+1}],$
for $k=0,...,2^m-1.$
Let $i,j \in \{0,...,2^m \}$  $i \leq j$ then for $s\in T_j^m$ and $t \in T_i^m $ we have
\begin{multline*}
  \pi_m(F)_t-\pi_m(F)_s -F_t+F_s\\
  =\begin{cases}
& (t-s)2^m[F_{t_{i+1}^m} -F_{t_i^m}] -[ F_t-F_s],~~\mbox{if}~~i=j\\
    & [t-t_j^m] 2^m[F_{t_{j+1}^m} - F_{t_j^m}]- [t_{i+1}^m -s]2^m[F_{t_{i+1}^m} -F_{t_i^m}] - [F_t-F_{t_i^m}]+ [F_{t_{i+1}^m}-F_s]~~\mbox{otherwise.}
\end{cases}
 \end{multline*}
Then,
\begin{multline*}
  \left|\pi_m(F)_t-\pi_m(F)_s -F_t+F_s\right|^p\le 4^{p-1}\times\\
   \begin{cases}
    &  \left[|t-s|^p2^{mp}|F_{t_{i+1}^m} -F_{t_i^m}|^p +| F_t-F_s|^p\right],~~\mbox{if}~~i=j\\
    &  \Big[[t-t_j^m]^p2^{mp} |F_{t_{j+1}^m} - F_{t_j^m}|^p+ |t_{i+1}^m -s|^p 2^{mp}|F_{t_{i+1}^m} -F_{t_i^m}|^p
    + |F_t-F_{t_i^m}|^p+ |F_{t_{i+1}^m}-F_s|^p \Big]~~\mbox{otw.}
  \end{cases}
  \end{multline*}
 Using Fubini to exchange the integration order, we have:
 \begin{align*}
     I^m&= 2 \int_0^1 \int_0^t \frac{\esp{|\pi_m(F)_t -\pi_m(F)_s -F_t +F_s|^p}}{|t-s|^{1+p\eta}}\dif s\dif t
     = I_1^m + I_2^m
     \end{align*}
     where
     \begin{align*}
     I_1^m&= \sum_{i=0}^{2^m-1} \int_{T_i^m} \int_{T_i^m, s<t}\frac{\esp{|\pi_m(F)_t -\pi_m(F)_s -F_t +F_s|^p}}{|t-s|^{1+p\eta}}\dif s\dif t,\\
     I^m_2&= 2\sum_{j=1}^{2^m-1}\sum_{i=0}^{j-1}  \int_{T_j^m} \int_{T_i^m}\frac{\esp{|\pi_m(F)_t -\pi_m(F)_s -F_t +F_s|^p}}{|t-s|^{1+p\eta}}\dif s\dif t.
     \end{align*}
In the integral $I_1^m$, $t$ and $s$ are in the same interval $T_k^m$, thus are at a distance at most $2^{-m}$, where as in $I_2^m$, $s$ and $t$ are not in the same interval $T_k^m.$ 
Note crucially that for $I^m_2$, the denominator in $|s-t|$ is not singular.

For $I_1^m$, since $t$ and $s$ are in the same $T_i^m$, we can split the expectation as:
\begin{eqnarray*}
\esp{ \Big|\pi_m(F)_t-\pi_m(F)_s -F_t+F_s \Big|^p} &\leq& 4^{p-1} \esp{|t-s|^p2^{mp}|F_{t_{i+1}^m} -F_{t_i^m}|^p +| F_t-F_s|^p}.
\end{eqnarray*}
Using the fact that $ {\mathbb E}[|F_t -F_s|^p]\leq |t-s|^{p/\alpha}$, we can bound $I_1^m \leq 4^{p-1}[I_{1,1}^m + I_{1,2}^m]$, where:
     \begin{align*}
         I_{1,1}^m&=\sum_{i=0}^{2^m-1}\int \int_{T_i^m \times T_i^m, s<t}\frac{|t-s|^p 2^{-m[\frac{p}{\alpha} -p]}}{|t-s|^{1+ p\eta}}\dif s\dif t=2\sum_{i=0}^{2^m-1}2^{-m[\frac{p}{\alpha} -p]}\int_{T_i^m} \int_{t_i^m}^t (t-s)^{p-p\eta-1}\dif s\dif t\\
         I_{1,2}^m&= \sum_{i=0}^{2^m-1}\int \int_{T_i^m \times T_i^m, s<t}\frac{|t-s|^\frac{p}{\alpha}}{|t-s|^{1+ p\eta}}\dif s\dif t= 2\sum_{i=0}^{2^m-1} \int_{T_i^m }\int_{t_i^m}^t|t-s|^{\frac{p}{\alpha}-p\eta-1}\dif s\dif t.
     \end{align*}
     Integrating in $s$, we get:
        \begin{align*}
         I_{1,1}^m&=2\sum_{i=0}^{2^m-1}2^{-m(\frac{p}{\alpha} -p)}\frac{1}{p(1-\eta)}\int_{T_i^m}  (t-t_i^m)^{p-p\eta}\dif t,\\
         I_{1,2}^m&= 2\sum_{i=0}^{2^m-1} \frac{1}{p( 1/\alpha-\eta )}\int_{T_i^m }|t-t_i^m|^{\frac{p}{\alpha} -p\eta}\dif t.
     \end{align*}
     Recall that  $T_i^m=[i2^{-m}, (i+1)2^{-m}]$, we get:
     \begin{align*}
         I_{1,1}^m&=2\sum_{i=0}^{2^m-1}2^{-m(\frac{p}{\alpha} -p)}\frac{1}{p(1-\eta)(p-p\eta+1)}2^{-m(p-p\eta+1)},\\
         I_{1,2}^m&= 2\sum_{i=0}^{2^m-1} \frac{1}{p(1/\alpha-\eta)(1-p\eta+p/\alpha)}2^{-m(\frac{p}{\alpha} -p\eta +1)}
     \end{align*}
     and
       \begin{align*}
         I_{1,1}^m&=2^{-m[\frac{p}{\alpha} -p]}\frac{2}{p(1-\eta)(p-p\eta+1)}2^{-m(p-p\eta)},\\
         I_{1,2}^m&= \frac{2}{p[-\eta + 1/\alpha](1-p\eta+p/\alpha)}2^{-m[\frac{p}{\alpha} -p\eta ]}.
     \end{align*}
Note that the crucial part is that we assumed $\alpha>1$ and $\eta< 1/\alpha$,
thus, the exponent is indeed negative. Now, we turn to the study of $I_2^m$. In
this case, $s$ and $t$ are not in the same interval $T_k^m,$  we can split the
expectation four ways:
\begin{align*}
\esp{\Big|\pi_m(F)_t-\pi_m(F)_s -F_t+F_s\Big|^p } \leq & 4^{p-1} \Big[[t-t_j^m]^p2^{mp} \esp{|F_{t_{j+1}^m} - F_{t_j^m}|^p} + \esp{|F_t-F_{t_i^m}|^p} \\
    &+ \esp{|F_{t_{i+1}^m}-F_s|^p} + |t_{i+1}^m -s|^p 2^{mp} \esp{|F_{t_{i+1}^m} -F_{t_i^m}|^p }.
\end{align*}

This prompts us to split $I_2^m$ four ways according to the above decomposition.
     \begin{align*}
         I_2^m\le 2\sum_{l=1}^4 I_{2,l}^m
     \end{align*}
For the first contribution above, and since $p>1$, we can use the
Lemma~\ref{ctrl:holder:process}. We get
     \begin{align*}
         I_{2,1}^m&=   2 \sum_{j=1}^{2^m-1}\sum_{i=0}^{j-1}\int \int _{T_j^m \times T_i^m} \frac{[t-t_j^m]^p\,2^{mp}\, \esp{|F_{t_{j+1}^m} - F_{t_j^m}|^p}}{|t-s|^{1+p\eta}}\dif s \dif t\\
         &\le 2 \sum_{j=1}^{2^m-1}\sum_{i=0}^{j-1}\int \int _{T_j^m \times  T_i^n} 2^{-m[\frac{p}{\alpha}- p]}\frac{|t-t_j^m|^p}{|t-s|^{1+p\eta}} \dif s\dif t.
     \end{align*}
We can now simply compute the remaining integrals:
     \begin{align*}
         I_{2,1}^m&\le2 \sum_{j=1}^{2^m-1}\int \int _{T_j^m \times [0,t_j^m]} 2^{-m[\frac{p}{\alpha}- p]}\frac{|t-t_j^m|^p}{|t-s|^{1+p\eta}} \dif s\dif t\\
         &\leq 2\frac{2^{-m[\frac{p}{\alpha}- p]}}{p\eta} \sum_{j=1}^{2^m-1}\int_{T_j^m} |t-t_{j}^m|^{p-p\eta}\dif t\\
         &\leq 2\frac{ 2^{-m[\frac{p}{\alpha}-p+p-p\eta+1-1]}}{p\eta(p-p\eta +1)} \leq \frac{2^{1-m[\frac{p}{\alpha} -p\eta]}}{p\eta(p-p\eta +1)} \cdotp
     \end{align*}
Again, we note that since $p \geq 1$ and $1/\alpha>\eta$, this exponent is indeed negative. 
For the second contribution, we proceed similarly, using \eqref{ctrl:holder:process} to estimate:
     \begin{align*}
         I_{2,2}^m &=   2 \sum_{j=1}^{2^m-1}\sum_{i=0}^{j-1}\int \int _{T_j^m \times T_i^m} \frac{\esp{|F_t-F_{t_i^m}|^p}}{|t-s|^{1+p\eta}}\dif s \dif t\\
         &\le 2 \sum_{j=1}^{2^m-1}\sum_{i=0}^{j-1}\int \int _{T_j^m \times T_i^m} \frac{|t-t_j^m|^{\frac{p}{\alpha}}}{|t-s|^{1+p\eta}} \dif s\dif t.
    \end{align*}
Here the sum of integrals in $d$ on $T_i^m$  yields an integral on $s$ between $0$ and $t_j^m.$  We can therefore estimate 
    \begin{align*}
        I_{2,2}^m &\leq \frac{2}{p\eta} \sum_{j=1}^{2^m-1}\int_{T_j^m} |t- t_j^m|^{\frac{p}{\alpha} -p\eta}\dif t \leq \frac{ 2^{1-m(\frac{p}{\alpha}-\eta)}}{p\eta( 1-p\eta +p/\alpha)} \cdotp
         \end{align*}
Now, for the terms involving $s \in T_i^m$, we use \eqref{ctrl:holder:process} to estimate:
        \begin{align*}
             I_{2,3}^m &= 2 \sum_{j=1}^{2^m-1}\sum_{i=0}^{j-1}\int \int _{T_j^m \times T_i^m} \frac{\esp{|F_{t_{i+1}^m}-F_s|^p}}{|t-s|^{1+p\eta}}\dif s \dif t\\
             &\le 2  \sum_{j=1}^{2^m-1}\sum_{i=0}^{j-1}\int \int _{T_j^m \times T_i^m} \frac{|s-t_{i+1}^m|^{\frac{p}{\alpha}}}{|t-s|^{1+p\eta}} \dif s\dif t.
             \end{align*}
Which gives after integration: 
        \begin{align*}
             I_{2,3}^m& \le  2  \sum_{i=0}^{2^m-2}\sum_{j=i+1}^{2^m-1 }\int \int_{T_j^m \times T_i^m} \frac{|s-t_{i+1}^m|^{\frac{p}{\alpha}}}{|t-s|^{1+p\eta}} \dif s\dif t\\
             &= 2  \sum_{i=0}^{2^m-2}\int_{T_i^m} \int_{t_{i+1}^m}^1 \frac{|s-t_{i+1}^m|^{\frac{p}{\alpha}}}{|t-s|^{1+p\eta}} \dif s\dif t\\
             &= \frac{2}{p\eta} \sum_{i=0}^{2^m-2}\int_{T_i^m}  |s-t_{i+1}^m|^{\frac{p}{\alpha}-p\eta}\dif s\leq \frac{2^{1-m[\frac{p}{\alpha} -p\eta ]}}{p\eta(1-p\eta +p/\alpha)}\cdotp
         \end{align*}
    The last term handled similarly:
    \begin{align*}
        I_{2,4}^m &= 2 \sum_{j=1}^{2^m-1}\sum_{i=0}^{j-1}\int \int _{T_j^m \times T_i^m} \frac{|t_{i+1}^m -s|^p 2^{mp} \esp{|F_{t_{i+1}^m} -F_{t_i^m}|^p}}{|t-s|^{1+p\eta}}\dif s\dif t\\
        &=   2^{1-m[\frac{p}{\alpha} -p]}\sum_{j=1}^{2^m-1}\sum_{i=0}^{j-1}\int \int_{T_j^m \times T_i^m} \frac{|s-t_{i+1}^m|^{p}}{|t-s|^{1+p\eta}} \dif s\dif t\\
        &=   2^{1-m[\frac{p}{\alpha} -p]}\sum_{i=0}^{2^m-1}\sum_{j=i+1}^{2^m-1}\int \int_{T_j^m \times T_i^m} \frac{|s-t_{i+1}^m|^{p}}{|t-s|^{1+p\eta}} \dif s\dif t\\
        &=   2^{1-m[\frac{p}{\alpha} -p]}\sum_{i=0}^{2^m-1}
        \int \int_{[t_{I+}^m,1]\times T_i^m} \frac{|s-t_{i+1}^m|^{p}}{|t-s|^{1+p\eta}} \dif s\dif t\\
         &\leq    2^{-m[\frac{p}{\alpha} -p]}\sum_{i=0}^{2^m-1}
         \int_{[t_{I+}^m,1]} \frac{|s-t_{i+1}^m|^{p-p\eta}}{p\eta }\dif s\\
         &= \frac{2^{1-m[\frac{p}{\alpha} -p\eta]}}{p\eta(p-p\eta +1)}\cdotp
         \end{align*}
         Then, there exists a constant $C$ depending on $p, \eta, \alpha$, such that
         \begin{align*}
             | I_{k,l}|^m \leq C\ 2^{-m[\frac{p}{\alpha}-1]},~~\mbox{for}~~(k,l)\in \{(1,2,(1,1)\}\cup\{(2,i),~~i=1,...,4\},~~\forall m
         \end{align*}

         The case of $J_n$ is simpler since
         \begin{align*}
           {\mathbb E}[ \sup_t |\pi_m(F)_t-F_t|^p]\leq C\  2^{-m\frac{p}{\alpha}}.
         \end{align*}
The proof is thus complete.
\end{proof}

\begin{remark}
    Since $I_m+J_m =  \esp{\Vert F-\pi_m(F)\Vert^p_{W_{\eta,p}}}$, we get the estimate:
    $$
     \esp{\Vert F-\pi_m(F)\Vert^p_{W_{\eta,p}}} \le   C\   \Big( 2^{-m\frac{p}{\alpha}} +2^{-m(1/\alpha -\eta)p}  \Big) \le  C\ 2^{-m\left(\frac{1}{\alpha} -\eta\right)p}.
    $$
\end{remark}

\subsection{The Interpolated Process Fits the setting of Lemma \ref{ctrl:norme:holder} }

In this section, we show how the interpolated random walk satisfies the assumptions of Lemma~\ref{ctrl:norme:holder}.


\begin{lemma}
Assume that for some $p <\alpha$,
\begin{equation}\label{moment:beta:fini}
 \sup_{n}\esp{ \left| \frac{Y_1+ \cdots + Y_n}{n^{1/\alpha}} \right|^p }<+\infty.
\end{equation}
Then, for some $C>0$, for all $n\ge 1$, we have:
$$
\esp{\Big| X_n(t) - X_n(s)\Big|^p} \le C\, |t-s|^{p/\alpha}.
$$
\end{lemma}


\begin{proof}

We start with the case where $t$ and $s$ are on the grid: $t=k2^{-n}$ and $s=j2^{-n}$.
Without loss of generality, assume $t > s$.
In this case, the increment of the interpolated process writes:
$$
X_n(t) - X_n(s) = \frac{1}{2^{n/\alpha}} \sum_{i=j}^{k-1} Y_i.
$$
From \eqref{moment:beta:fini}, we deduce that for all $n\in \N$, there exists some constant $C>0$ such that :
$$
\esp{ \Big| Y_1 + \cdots + Y_n \Big|^p} \le C n^{p/\alpha}.
$$
Hence, we directly get:
\begin{eqnarray*}
\esp{ |X_n(t) - X_n(s)|^p } &=& \esp{ \left|\frac{1}{2^{n/\alpha}} \sum_{i=j}^{k-1} Y_i \right|^p}\\
&\le &C \frac{1}{2^{np/\alpha}} |k-j|^{p/\alpha}\\
&=& C\left| \frac{k}{2^{n}}-\frac{j}{2^{n}} \right|^ {p/\alpha} = C|t-s|^{p/\alpha},
\end{eqnarray*}
thus, the moment condition in lemma \ref{ctrl:norme:holder} is satisfied for points on the grid.

Now, let us discuss the case where $t$ and $s$ are not on the grid, but in the same sub-interval $[j2^{-n}, (j+1)2^{-n}]$.
In that case, the increment of the interpolated process reduces to 
$$
X_n(t) - X_n(s) = \frac{1}{2^{n/\alpha}} Y_j\sqrt{2^n} (t-s).
$$
In that case, we get:
$$
\esp{ |X_n(t) - X_n(s)|^p }  = \frac{\sqrt{2^{p n}} (t-s)^p}{2^{np/\alpha}} \ \esp{|Y_j|^p}.
$$
Since  $ \alpha\le 2$, we have $\frac{p}{2} \le \frac{p}{\alpha}$, thus $(2^n)^{p/2} \le (2^n)^{p/\alpha}$. Hence, we get:
$$
\esp{ |X_n(t) - X_n(s)|^p }  \le C (t-s)^p \le C (t-s)^{p/\alpha},
$$
the last inequality being true since $t$ and $s$ are in the same sub-interval, hence $t-s \le 1$.

It remains us to discuss the case where $t$ or $s$ are not on the grid, and in separate sub-intervals. In that case, we introduce $s^+$ and $t^-$, on the grid $\{k2^{-n}, k=1,\dots, 2^n\}$ such that $s \le s^+ \le t^- \le t$.
We can always split the difference into:
$$
X_n(t) - X_n(s) = \Big(X_n(t) - X_n(t^-)\Big) + \Big(X_n(t^-) - X_n(s^+)\Big) + \Big(X_n(s^+) - X_n(s) \Big).
$$

Now, $t$ and $t^-$ are in the same sub-interval, and so are $s$ and $s^+$. Moreover, $s^+$ and $t^-$ are on the grid, hence, we can directly use the controls explained above to get:
$$
\esp{ |X_n(t) - X_n(s)|^p }  \le C \Big(  (t-t^-)^{p/\alpha} +   (t^- - s^+)^{p/\alpha}  +  (s^+-s)^{p/\alpha} \Big) \le C_p (t-s)^{p/\alpha},
$$
where  for the last inequality, we used the fact that $a^p + b^p \le C_p (a+b)^p$, that holds with $C_p = 2^{1-p}$ if $p\le 1$ and $C_p= 1$ if $p \ge 1$.
Hence, Lemma \ref{ctrl:norme:holder} is applicable and the norm in $W_{\eta,p}$ is controlled.

\end{proof}


\subsection{Derivation of the Final Rate of convergence}
\label{final:derivation}
In this section, we derive the final rate of convergence of the interpolated
random walk $(X_n(t))_{t\in [0,1], n \ge 0 }$ towards the stable process $S$ in
$W_{\eta,p}$. Recall:
\begin{equation*}
  \dist_{W_{\eta,p}}(X_n,S)
  = \sup_{F\in \Lip_{1}}\Big(\esp{F(X_n)}- \esp{F(S)} \Big).
\end{equation*}
Let $F$ be a Lipschitz function from $W_{\eta,p}$ to $\R$. We have:
\begin{eqnarray*}
  \esp{F(X_n) -F(S)} &=&  \Big(\esp{F(X_n)} -\esp{F(\pi_m(X_n))} \Big) +  \Big(\esp{F(\pi_m(X_n))}-\esp{F(\pi_m(S))} \Big)\\
                           &&+  \Big(\esp{F(\pi_m(S))} -\esp{F(S)}\Big)\\
                           &\le& \esp{\Vert X_n - \pi_m(X_n)\Vert_{W_{\eta,p}}} + \dist_W\Big(\pi_m(X_n),\pi_m(S)\Big)+ \esp{\Vert S - \pi_m(S)\Vert_{W_{\eta,p}}}.
\end{eqnarray*}
Thus,
\begin{eqnarray*}
  \dist_{W_{\eta,p}}(X_n,S) &=& \sup_{F \in {\Lip}(W_{\eta,p}) }\esp{F(X_n)} -\esp{F(S)} \\
             & \le& \esp{\Vert X_n - \pi_m(X_n)\Vert_{W_{\eta,p}} }+ \dist_{W_{\eta,p}}(\pi_m(X_n),\pi_m(S))+ \esp{ \Vert S - \pi_m(S)\Vert_{W_{\eta,p}}}.
\end{eqnarray*}
We plug-in each the controls we obtained for each term, to get:
$$\dist_{W_{\eta,p}}(X_n,S)  \le   C\ \Big(2^{-m\left(\frac{1}{\alpha} -\eta\right)p} + 2^{m} 2^{n\frac{\alpha-2}{\alpha}} +2^m 2^{-n\frac{\gamma}{\alpha}} +2^{-m\left(\frac{1}{\alpha} -\eta\right)p} \Big)
$$
Crucially, we can already notice we can choose  $m$ and $n$ such that every exponent is actually negative.
Up to a modification of the constant, we can group the first and fourth term together, to get an estimate:
\begin{equation}\label{borne:wasserstein:avec:m}
\dist_{W_{\eta,p}}(X_n,S)\le C\ \Big(2^{-m\left(\frac{1}{\alpha} -\eta\right)p} + 2^{m} 2^{-n\left(\frac{2}{\alpha}-1\right)} +2^m 2^{-n\frac{\gamma}{\alpha}} \Big).
\end{equation}
We see that we need to choose $m$ in terms of $n$ to ensure that:
\begin{itemize}
    \item $2^{-m\left(\frac{1}{\alpha} -\eta\right)p}$ tends to 0 as fast as possible, that is $m$ as large as possible,
    \item $2^{m} 2^{-n\left(\frac{2}{\alpha}-1\right)}$ tends to 0 as fast as possible, that is $m \le n\left(\frac{2}{\alpha}-1\right)$.
    \item $2^m 2^{-n\frac{\gamma}{\alpha}} $ tends to 0 as fast as possible, that is $m \le n\frac{\gamma}{\alpha}$.
\end{itemize}
One can solve this optimization problem to find the optimal $m$ in terms of $n$, but the calculations are quite heavy. 
Hence, we provide a more tractable calculation giving a sub-optimal rate. First, note that 
$$
\dist_{W_{\eta,p}}(X_n,S)\le C\ \Big(2^{-m\left(\frac{1}{\alpha} -\eta\right)p} + 2^{m-n \times \min\left(\frac{2}{\alpha}-1, \frac{\gamma}{\alpha}\right)}  \Big)
$$
We set $m=\kappa n$, and find $\kappa$ such that the exponent match:
$$
-\kappa \left( \frac{1}{\alpha} - \eta \right)p = \kappa - \min\left( \frac{2}{\alpha} -1 , \frac{\gamma}{\alpha} \right) \Rightarrow \kappa = \frac{\min\left( \frac{2}{\alpha} -1 , \frac{\gamma}{\alpha} \right)}{1+ \left( \frac{1}{\alpha} - \eta \right)p}.
$$

Our final bound comes out to be:
$$
\dist_W(X_n,S)\le C\  2^{-n \upsilon},
$$
with 
$$
\upsilon = 
\left(\frac{1}{\alpha}-\eta\right)p 
\frac{\min\left(\frac{2}{\alpha}-1, \frac{\gamma}{\alpha} \right)}{1+\left(\frac{1}{\alpha}- \eta\right)p}.
$$

  \section{Proof of the Moment Condition}
  \label{section:moment:cond:proof}
  In this section, we prove that the moment condition \ref{moment:beta:fini} actually holds  for all $p <\alpha$.
Recall that a random variable $Y$ is  in the normal domain of attraction of a stable distribution if its distribution function $F_Y$ is of the form:
$$
1-F_Y(t) = \P(Y \ge t) = \frac{A+ \varepsilon(t)}{t^\alpha} \mbox{ and } F_Y(-t) = \frac{A+ \varepsilon(-t)}{(-t)^\alpha},
$$
whenever $|t| \ge 1$.  
We split this paragraph into two parts in order to isolate the main ideas.
We first deal with the symmetrized Pareto case in subsection \ref{sec:pareto:moment}, then deal with the full case in subsection~\ref{sec:normal:domain}.
In both case, the proof uses a technique very reminiscent of Stein's method. 
The key is  to perform an integration by parts with respect to the Pareto distribution.

\subsection{The case of Symmetrized Pareto}
\label{sec:pareto:moment}

\begin{theorem}\label{moment:ordre:p}
    Let $Y_1,\dots, Y_n$ independent and identically distributed with probability density function:
    $$\mathbb{P}(Y\in \dif y) =  \frac{\alpha}{2}\frac{\dif y}{|y|^{\alpha+1}}\ind{|y| >1} .$$
    Then, it holds that for all $p<\alpha$:
    $$
 \sup_{n}\    \mathbb{E}\left[\left| \frac{1}{n^{1/\alpha}}\sum_{i=1}^{n} Y_i\right|^p\right] <+\infty.
    $$
\end{theorem}

This result is of independent interest in the sense that, to the best of our knowledge, this estimate is not present in the literature.
This result is not surprising however, the Pareto distributions being in the domain of attraction of the stable distribution, similar estimate are expected to hold. 

Similarly to the non-integrable case, the idea is to relate the moment of $\frac{1}{n^1/\alpha} \sum Y_i$ to the moment of a stable random variable. 
However, the lack of uniform Lipschitz property of $x\mapsto |x|^p$ at 0 prevents us from using the results of \cite{chen:nourdin:xu:2018} directly. 
We introduce the function $\phi_p$:
\begin{equation}\label{regularisation:puissance}
\phi_p(x) = \begin{cases}
|x|^p & \mbox{ if } |x| >1\\
\frac{p}{2} x^2 + (1-\frac{p}{2}) & \mbox{else}.
\end{cases}
\end{equation}
This function interpolates a parabola close to the origin and the function $x\mapsto |x|^p$.
By construction, it holds that 
\begin{equation}
\label{ineg:phi:p:puissance}
|x|^p \le 1+ \phi_p(x), 
\end{equation}
and $\phi_p \in {\cal C}^{2,p}_b$, meaning  twice differentiable with $p$-H\"older  first derivative and bounded second derivative.
We will use this function to establish the moment condition. Namely, we prove:

\begin{equation}\label{objectif:estimate}
    \E\left[\phi_p\left(\frac{1}{n^{1/\alpha}}\sum_{j=1}^n Y_i\right)\right] <+\infty.
\end{equation}

Estimate \eqref{objectif:estimate} coupled with  inequality \eqref{ineg:phi:p:puissance} on $\phi_p$ yield the moment condition of Theorem~\ref{moment:ordre:p}. 
We now focus of establishing \eqref{objectif:estimate}.





\begin{lemma}\label{calcul:esperance:IPP}
    Let $G \in \mathcal{C}^{2,p}_b(\R).$ 
    Let $Y$ with probability density function 
    $$\mathbb{P}(Y\in \dif y) =  \frac{\alpha}{2}\frac{\dif y}{|y|^{\alpha+1}}\ind{|y| >1} .$$
    It holds that 
    $$
    \frac{2}{\alpha}\esp{ G'(Y)Y } = LG(0) - \Big( G(1) + G(-1) -2 G(0) \Big) -\int_{|y| \le 1}\Big( G(y) - G(0) - G'(0) y \Big) \frac{\alpha \dif y}{|y|^{\alpha+1}} ,
    $$
    where $L$ is the non local operator:
    $$
    L\varphi(x) = \int_{-\infty}^{+\infty} \Big(\varphi(x+y)  - \varphi(x) - \varphi'(x) y \Big)  \frac{\alpha \dif y}{|y|^{1+\alpha}}.
    $$
\end{lemma}

\begin{proof}
We start by expressing the left hand side using the pdf of $Y$:
\begin{eqnarray*}
\esp{ G'(Y)Y } &=& \int_{|y| >1} G'(y)y   \frac{\alpha}{2}\frac{\dif y}{|y|^{\alpha+1}}\ind{|y| >1} \\
&=& \int_{1}^{+\infty} G'(y)y   \frac{\alpha}{2}\frac{\dif y}{y^{\alpha+1}} + \int_{-\infty}^{-1} G'(y)y   \frac{\alpha}{2}\frac{\dif y}{(-y)^{\alpha+1}} \\
&=&E_+ + E_-.
\end{eqnarray*}

By integration by parts, we have:
\begin{eqnarray*}
E_+ &=& \int_{1}^{+\infty} G'(y)   \frac{\alpha}{2}\frac{\dif y}{y^{\alpha}} =  \frac{\alpha}{2}\left[ \frac{G(y) - G(0)}{y^\alpha} \right]_1^{+\infty} -  \frac{\alpha}{2}\int_{1}^{+\infty} \Big( G(y) - G(0) \Big) \frac{(-\alpha)}{y^{\alpha+1}}\dif y\\
&=& \frac{\alpha}{2} \left( \alpha\int_{1}^{+\infty} \Big( G(y) - G(0) \Big) \frac{1}{y^{\alpha+1}}\dif y  - \Big( G(1)-G(0) \Big)  \right) .
\end{eqnarray*}

Similarly, for the negative side, we have:
\begin{eqnarray*}
E_- &=& -\int_{-\infty}^{-1} G'(y)   \frac{\alpha}{2}\frac{\dif y}{(-y)^{\alpha}} = -\frac{\alpha}{2}\left[ \frac{G(y) - G(0)}{(-y)^\alpha} \right]_{-\infty}^{-1} + \frac{\alpha}{2}\int_{-\infty}^{-1} \Big( G(y) - G(0) \Big) \frac{(-1)(-\alpha)}{(-y)^{\alpha+1}}\dif y\\
&=&\frac{\alpha}{2} \left( \alpha\int_{-\infty}^{-1} \Big( G(y) - G(0) \Big) \frac{1}{(-y)^{\alpha+1}}\dif y - \Big(G(-1) - G(0) \Big)\right).
\end{eqnarray*}

We can group the two integrals together writing:
$$
\int_{-\infty}^{-1} \Big( G(y) - G(0) \Big) \frac{(-1)(-\alpha)}{(-y)^{\alpha+1}}\dif y + \int_{-\infty}^{-1} \Big( G(y) - G(0) \Big) \frac{1}{(-y)^{\alpha+1}}\dif y = \int_{|y| >1}\Big( G(y) - G(0) \Big) \frac{\dif y}{|y|^{\alpha+1}}.
$$
Besides, since the distribution is symmetric, we can add $G'(0)y$ under the integral without changing its value:
$$
 \alpha\int_{|y| >1}\Big( G(y) - G(0) \Big) \frac{\dif y}{|y|^{\alpha+1}} =  \int_{|y| >1}\Big( G(y) - G(0) - G'(0) y \Big) \frac{\alpha \dif y}{|y|^{\alpha+1}}.
$$
Adding the integral for $|y|\le 1$, we get:
$$
\alpha\int_{|y| >1}\Big( G(y) - G(0) \Big) \frac{\dif y}{|y|^{\alpha+1}}  =
LG(0) -  \int_{|y| \le 1}\Big( G(y) - G(0) - G'(0) y \Big) \frac{\alpha \dif y}{|y|^{\alpha+1}}.
$$
Consequently, we finally have:
$$
\esp{ G'(Y)Y } =  \frac{\alpha}{2} \left( LG(0) - \Big( G(1) + G(-1) -2 G(0) \Big) - \int_{|y| \le 1}\Big( G(y) - G(0) - G'(0) y \Big) \frac{\alpha \dif y}{|y|^{\alpha+1}} \right).
$$
\end{proof}
\begin{proof}[Proof of Theorem~\protect\ref{moment:ordre:p}]
Consequently, we see that we can bound $\esp{ G'(Y)Y }$ with the second derivative of $G$. Let us now define some notations.
Recall the function $\phi_p$ defined in\eqref{regularisation:puissance} above. 
We use bold letters to denote vectors in $\R^n$.
Besides, capital letters will denote random variables. 

Let $\textbf{Y} = (Y_1, \dots, Y_n)$ be a vector with IID entries, such that
$$\mathbb{P}(Y_i\in \dif y) =  \frac{\alpha}{2}\frac{\dif y}{|y|^{\alpha+1}}\ind{|y| >1} .$$

Let also $\textbf{S} = (S_1, \dots, S_n)$ be a stable process in $\R^n$ with independent coordinates.  For any function $F\in \mathcal{C}^{2,p}_b$, it holds that
\begin{equation}\label{Stein:Equation:multidim}
\esp{F(\textbf{Y})} - \esp{F(\textbf{S}) }= - \int_0^{+\infty} \esp{ {\cal L}P_t F(\textbf{Y})}\dif t,
\end{equation}
where ${\cal L}$ is the generator of the stable driven Ornstein-Uhlenbeck process:
$$
\dif\textbf{X}_t = -\frac{1}{\alpha} \textbf{X}_t \dif t + \dif \textbf{S}_t,
$$

We note here that the generator ${\cal L}$ has the expression:
$$
{\cal L} \varphi(\textbf{x}) = -\frac{1}{\alpha} \nabla \varphi(\textbf{x}) \cdot \textbf{x} + \sum_{i=1}^n \int_{\R} \Big(F(\textbf{x} + u \textbf{e}_i) - F(\textbf{x}) - \nabla F(\textbf{x}) \cdot u \Big)\frac{1}{|u|^{\alpha +1}}\dif u,
$$
where $\textbf{e}_i$ is the canonical basis element of $\R^n$.

We use equation \ref{Stein:Equation:multidim} with 
$$F(\textbf{x}) = \phi_p\left(\frac{x_1 + \dots + x_n}{n^{1/\alpha}}\right).$$ 
This leads us to the identity:
\begin{equation}\label{stein:prelim:phip}
\esp{\phi_p\left(\frac{Y_1 + \dots + Y_n}{n^{1/\alpha}}\right) } -
\esp{\phi_p\left(\frac{S_1 + \dots + S_n}{n^{1/\alpha}}\right) }=- \int_0^{+\infty} \esp{ {\cal L}P_t F(\textbf{Y})}\dif t.
\end{equation}
Now, since $\textbf{S} = (S_1, \dots, S_n)$ be a stable process in $\R^n$, it has finite moments of order $p<\alpha$.
Thus, equation \eqref{objectif:estimate} will hold as soon as we control the right hand side of \eqref{stein:prelim:phip}.
Hence, we need to control:
$$
E=\sum_{i=1}^n\esp{ -\frac{1}{\alpha}\partial_{y_i} G(\textbf{Y}) Y_i + \int_{\R} \Big(G(\textbf{Y} + u \textbf{e}_i) - G(\textbf{Y}) - \nabla G(\textbf{Y}) \cdot u \Big)\frac{1}{|u|^{\alpha +1}}\dif u },
$$
where $G=P_tF$. Let us denote for $i=1,\dots,n$ 
$$
{G}_{\backslash Y_i}(y) 
= G\Big(Y_1,\dots, Y_{i-1},y,Y_{i+1},\dots, Y_n\Big).
$$
In other words, ${G}_{\backslash Y_i}$ is the (random) function $G$ where we replace the $i$-th component by the variable $y\in \R$. 
By conditioning in the sum by $Y_1,\dots, Y_{i-1},Y_{i+1},\dots, Y_n$, it holds that
$$
E= \sum_{i=1}^n\esp{ -\frac{1}{\alpha} {G}_{\backslash Y_i}'(Y_i) Y_i + \int_{\R} \Big({G}_{\backslash Y_i}(Y_i + u) - {G}_{\backslash Y_i}(Y_i) - {G}_{\backslash Y_i}'(Y_i) \cdot u \Big)\frac{1}{|u|^{\alpha +1}}\dif u }.
$$
These notations allows us to reduce to the once dimensional case, and use Lemma \ref{calcul:esperance:IPP} above.
Namely, we deduce that:
\begin{multline*}
\esp{  {G}_{\backslash Y_i}'(Y_i) Y_i}=
\textbf{E}\left[\vphantom{\int_{|y| \le 1}}\frac{\alpha}{2} \left( L{G}_{\backslash Y_i}(0) - \Big( {G}_{\backslash Y_i}(1) + {G}_{\backslash Y_i}(-1) -2 {G}_{\backslash Y_i}(0) \Big) \right.\right.\\
 \left.\left.- \int_{|y| \le 1}\Big( {G}_{\backslash Y_i}(y) - {G}_{\backslash Y_i}(0) - {G}_{\backslash Y_i}'(0) y \Big) \frac{\alpha \dif y}{|y|^{\alpha+1}} \right) \right].
\end{multline*}
Recall $G=P_t F$, we need to control
\begin{multline*}
\esp{ {\cal L} P_tF(\textbf{Y}) }= \sum_{i=1}^n \esp{-\frac{1}{\alpha} {G}_{\backslash Y_i}'(Y_i) Y_i
+ LG(Y_i)}\\
= \sum_{i=1}^n
\E \left(\frac{1}{2} \left( -L{G}_{\backslash Y_i}(0) +  \Big( {G}_{\backslash Y_i}(1) + {G}_{\backslash Y_i}(-1) -2 {G}_{\backslash Y_i}(0) \Big) \right.\right.\\
 \left.\left.+ \int_{|y| \le 1}\Big( {G}_{\backslash Y_i}(y) - {G}_{\backslash Y_i}(0) - {G}_{\backslash Y_i}'(0) y \Big) \frac{\alpha \dif y}{|y|^{\alpha+1}}\right) + LG(Y_i)  \right).
\end{multline*}

Looking at this identity, it appears that the quantity we need to focus on is
$L{G}_{\backslash Y_i}(Y_i) - L{G}_{\backslash Y_i}(0) $, all other terms can be
controlled by $\Vert G'' \Vert_{\infty}$ from a Taylor's expansion.
We have:
\begin{multline*}
L{G}_{\backslash Y_i}(Y_i) - L{G}_{\backslash Y_i}(0) =  \int_{-\infty}^{+\infty} \Big( {G}_{\backslash Y_i}(Y_i+y) - {G}_{\backslash Y_i}(Y_i) - {G}_{\backslash Y_i}'(Y_i) y \Big) \frac{\alpha \dif y}{|y|^{\alpha+1}}\\
    -\int_{-\infty}^{+\infty} \Big( {G}_{\backslash Y_i}(y) - {G}_{\backslash Y_i}(0) - {G}_{\backslash Y_i}'(0) y \Big) \frac{\alpha \dif y}{|y|^{\alpha+1}}\cdotp
\end{multline*}
We split this integral between $\{|y| \le 1\}$, where we can do a Taylor's expansion, and $\{|y| > 1\}$:
\begin{multline*}
L{G}_{\backslash Y_i}(Y_i) - L{G}_{\backslash Y_i}(0) \\
   \shoveleft{ =  \int_{\{|y| \le 1\}} \Big( \big({G}_{\backslash Y_i}(Y_i+y)-{G}_{\backslash Y_i}(y)\big) - \big({G}_{\backslash Y_i}(Y_i)-{G}_{\backslash Y_i}(0) \big) - \big({G}_{\backslash Y_i}'(Y_i)- {G}_{\backslash Y_i}'(0)\big) y \Big) \frac{\alpha \dif y}{|y|^{\alpha+1}}}\\
\shoveright{     + \int_{\{|y| > 1\}} \Big( \big({G}_{\backslash Y_i}(Y_i+y)-{G}_{\backslash Y_i}(y)\big) - \big({G}_{\backslash Y_i}(Y_i)-{G}_{\backslash Y_i}(0) \big) - \big({G}_{\backslash Y_i}'(Y_i)- {G}_{\backslash Y_i}'(0)\big) y \Big) \frac{\alpha \dif y}{|y|^{\alpha+1}}}\\
    = I +I\!\!I.
\end{multline*}

We now turn to the second integral above. First, noticing that on $\{|y| > 1\}$, the measure $\frac{\alpha \dif y}{|y|^{\alpha+1}}$ is finite and symmetric, we can cancel the compensation term. Next, we have:
\begin{eqnarray*}
{G}_{\backslash Y_i}(Y_i+y)-{G}_{\backslash Y_i}(y) &=& \int_0^1 {G}_{\backslash Y_i}'(y+ \theta Y_i)\cdot Y_i  \dif\theta,\\
{G}_{\backslash Y_i}(Y_i)-{G}_{\backslash Y_i}(0) &=& \int_0^1 {G}_{\backslash Y_i}'(0+\theta Y_i)\cdot Y_i \dif \theta.
\end{eqnarray*}

This gives us:
$$
I\!\!I =  \int_{\{|y| > 1\}} \int_0^1 \Big({G}_{\backslash Y_i}'(y+ \theta Y_i)- {G}_{\backslash Y_i}'(\theta Y_i) \Big) Y_i \dif\theta  \frac{\alpha \dif y}{|y|^{\alpha+1}}.
$$
Now, we do an additional Taylor's expansion, writing that
$$
{G}_{\backslash Y_i}'(y+ \theta Y_i)- {G}_{\backslash Y_i}'(\theta Y_i)= \int_0^1 {G}_{\backslash Y_i}''(\theta Y_i + \mu y) y d\mu, 
$$
which can be bounded,  since $\ind{|y| > 1}\frac{\alpha \dif y}{|y|^{\alpha+1}}$ integrates $|y|$ at infinity. This gives us the following estimate:
$$
\esp{I\!\!I } \le \Vert G'' \Vert_\infty \esp{|Y_1|}.
$$
\begin{remark}
    We point out that this last estimate is rather tricky. 
    Doing bluntly a second order Taylor estimation on ${G}_{\backslash Y_i}$ would not work here because $y^2$ is not integrable at infinity against  $\frac{\alpha \dif y}{|y|^{\alpha+1}}$. Instead, here, the second increment $y$ is replaced by $Y_i$, which can be estimated.
\end{remark}

For the sake of completeness, let us write the estimate we obtain for $I$:
$$
\esp{I} \le \Vert G'' \Vert_\infty \frac{1}{2-\alpha}\cdotp
$$

Hence, we managed to relate every terms in \eqref{Stein:Equation:multidim} to the derivatives of $G$. 
We have the following lemma:
\begin{lemma}
    For any $F\in {\cal C}^{2,p}_b$, it holds that
    \begin{itemize}
        \item $\nabla P_t F(x) \le e^{-t/\alpha} P_t(\nabla F)(x)$.
        \item If $\nabla F$ is Lipschitz, then $\nabla P_tF$ is Lipschitz, with
        $$|\nabla P_t F(x) - \nabla P_t F(y) |\le C e^{-t/ \alpha}|x-y|^{\beta-1}.$$
    \end{itemize}
\end{lemma}

Recalling $G=P_tF$, derivatives of $G$ actually yields an additional $1/n^{1/\alpha}$ factor. 
Hence, our final estimate is the following:

\begin{multline*}
\esp{ {\cal L} P_tF(\textbf{Y}) } = \sum_{i=1}^n \esp{-\frac{1}{\alpha} {G}_{\backslash Y_i}'(Y_i) Y_i
+ LG(Y_i)}\\
= \sum_{i=1}^n
\textbf{E}\left[\frac{1}{2} \left( -L{G}_{\backslash Y_i}(0) +  \Big( {G}_{\backslash Y_i}(1) + {G}_{\backslash Y_i}(-1) -2 {G}_{\backslash Y_i}(0) \Big) \right.\right.\\
 \left.\left.+ \int_{|y| \le 1}\Big( {G}_{\backslash Y_i}(y) - {G}_{\backslash Y_i}(0) - {G}_{\backslash Y_i}'(0) y \Big) \frac{\alpha \dif y}{|y|^{\alpha+1}}\right) + LG(Y_i)  \right].
\end{multline*}

Finally, in this sum, we bound each term by $\Vert G'' \Vert_\infty = \frac{1}{n^{2/\alpha}} e^{-t/\alpha} \Vert \phi_p''\Vert_\infty $. Now recall that $\phi_p$ interpolates between a 2nd order polynomial and $x\mapsto x^p$, with $p\le 2$, hence, its second derivative is bounded.

This gives us the upper bound:
\begin{equation}\label{lem-maj-L}
|E| \le C \sum_{i=1}^n  \frac{1}{n^{2/\alpha}} e^{-t/\alpha} \Vert \phi_p''\Vert_\infty= C n^{1-\frac{2}{\alpha}} e^{-t/\alpha}.
\end{equation}

The presence of the exponential term allows us to integrate in $t$ from $0$ to
$\infty$, and finally, the proof is complete.
\end{proof}
Note that we also established the following lemma:
\begin{lemma}\label{lem-maj-L}
Let $G$ be a $C^{2,p}_b$ function with bounded derivative then,
\begin{align}
\left|\int_{|u| \leq 1} [G(u+y)-G(y) -uG'(y)]\frac{1}{|u|^{\alpha+1}}\dif u\right| &\leq \frac{2}{2-\alpha} \|G"\|_{\infty},\\
\left| L(G)(y) -L(G)(0)\right|& \leq C_{\alpha} \|G"\|_{\infty} [1+|y|],~~\forall y \in {\mathbb R}.
\end{align}
\end{lemma}

\subsection{In the Normal domain of attraction}
\label{sec:normal:domain}
In this paragraph, we demonstrate how to establish the moment condition \eqref{moment:beta:fini} for a general random variable in a Normal Stable domain of attraction. 
We recall that for a random variable $Y$ to be in the normal domain of attraction of a stable distribution, we require its distribution function $F_y$ to satisfy:
$$
1-F_Y(t) = \P(Y \ge t) = \frac{A+ \varepsilon(t)}{t^\alpha} \mbox{ and } F_Y(-t) = \frac{A+ \varepsilon(-t)}{(-t)^\alpha},
$$
whenever $|t| \ge 1$.  
From that decomposition, we see that we can extract a Pareto component from $F_Y$.
We let $\P_Y(\dif y) = \eta \P_Z(\dif y)+ (1-\eta)\mu(\dif y)$, where $P_Z$ is the distribution of a symmetrized Pareto.

\begin{remark}
    The constant $\eta$ is chosen according to $A$ and $\alpha$ to ensure that we indeed define a probability distribution.
    In the case where the constant $\eta=1$, notice that we recover the previous symmetrized Pareto distribution. 
    This constant $\eta$ will also play a role in the generator below.
\end{remark}

Observe now that $\mu$ is simply a signed measure, for which we have the following information:
\begin{itemize}
    \item $\mu$ is centered, signed and finite
    \item $\mu(t;+\infty)= \dfrac{\varepsilon(t)}{t^\alpha}$, if $t\ge 1$, and
    $\mu(-\infty; -t]= \dfrac{\varepsilon(-t)}{(-t)^\alpha}$, it $t\le -1$
    \item $|\mu| \leq {\mathbb P}_Y + {\mathbb P}_Z \leq (1+ \|\varepsilon\|_{\infty}){\mathbb P}_Z$
\end{itemize}
\begin{prop}\label{pro-maj-gen}
For all $1<\alpha <\beta$ there exists a constant $c_{\alpha,\beta,\eta}$ such that for all 
 $Y$ a random variable with distribution in ${\mathcal D}^{\alpha,\eta}$ and $G \in {\mathcal F}^{\beta}$ and $A>1$
\begin{align*}
\esp{ {\mathcal L}^{\alpha, \eta}G(Y) }\leq c_{\alpha,\beta,\eta}
\left[ \|G"\|_{\infty} 
+ (1-\eta)(1+\|\varepsilon\|_{\infty})\left\{A^{\beta-\alpha} \|G'\|_{\beta-1,\text{Hol}} + A^{2-\alpha} \|G"\|_{\infty}\right\}
\right].
\end{align*}
\end{prop}
\begin{proof}[Proof of Proposition \ref{pro-maj-gen}]
The main lines of the proof is similar, we first derive from Stein's equation:
$$
\esp{F(Y)}-\esp{F(S)} = \int_0^t \esp{{\cal L} P_t F(Y)}\dif t,
$$
where ${\cal L}$ is the generator of the Orstein Uhlenbeck process:
\begin{eqnarray*}
    {\cal L} \varphi(x) &=& -\frac{1}{\alpha} \varphi'(x)\cdot x + \int_{-\infty}^{+\infty}\Big( \varphi(x+y) - \varphi(x)-\varphi'(x)y \Big) \frac{\eta \alpha}{2} \frac{1}{|y|^{1+\alpha}}\dif y\\
    &=&  -\frac{1}{\alpha} \varphi'(x)\cdot x + \eta L\varphi(x)
\end{eqnarray*}
\begin{remark}
    We point out here the coefficient $\eta$ that was not present or rather equal to one in the Pareto case. 
    We can put this coefficient here, up to a modification of the Lévy measure of the stable process, and this saves a renormalisation by some coefficient $\sigma$ in Nourdin \cite{chen:nourdin:xu:2018}.
\end{remark} 

Let $Y$ be a random variable with distribution in ${\mathcal F}^{\alpha,\eta}$ and $Z$ be a random variable with symmetric Pareto distribution and $F_Y,$ $F_Z$ their distribution function. Then $F_Y -\eta F_Z$ is a function of finite bounded variation associated to a signed measure $\mu$ such that
\begin{align*}
& |\mu| \leq {\mathbb P}_Y + {\mathbb P}_Z \leq (1+ \|\varepsilon\|_{\infty}){\mathbb P}_Z,\\
&\mu( ]t,+\infty[)=\frac{\varepsilon(t)}{t^{\alpha}}~~\forall ~~t>1,\\
&\mu( ]-\infty,t])=\frac{\varepsilon(t)}{t^{\alpha}}.
\end{align*}
With these notations in hand, since $Y$ is a centered random variable
\begin{align*}
\esp{ {\mathcal L}( G) (Y)} =& -\frac{1}{\alpha} \int_{{\mathbb R}}y\left[G'(y) -G'(0)\right] \left[ \eta\dif {\mathbb P}_Z(y) + (1-\eta) d \mu(y) \right]\\
&+ \eta \int_{{\mathbb R} } L(G) (y)\dif {\mathbb P}_Y(y).
\end{align*}
Using Lemma \ref{calcul:esperance:IPP} for the integral with respect to $d{\mathbb P}_Z$ we obtain
\begin{align*}
\esp{ {\mathcal L}( G) (Y) }&= -\frac{\alpha}{2}\eta \left[ G(1) +G(-1)- 2 G(0)\right] \\
& \frac{1}{\alpha}\eta \int_{|u| \leq 1 }\left[ G(u) -G(0) -uG'(0)\right] \frac{\dif u}{|u|^{\alpha +1}}\\
&+ \eta \int_{{\mathbb R}} \left[ L(G)(y) -L(G)(0) \right]\dif {\mathbb P}_Y(y)\\
&-\frac{1}{\alpha} \int_{{\mathbb R}}y\left[G'(y) -G'(0)\right] (1-\eta)d \mu(y) .
\end{align*}
Using Lemma \ref{lem-maj-L}
the three first terms of the right member are bounded by 
$C_{\alpha, \eta} \|G"\|_{\infty}$ where $C_{\alpha,\eta}$ depends only on $\alpha$ and $\eta.$ 

For the last term, notice how when $\eta=1$, then term is not present, and the result follows from the Pareto case. 
When $\eta <1$, we shall use the following Lemma to bound the last term by for some $A>1$ to be chosen later.
\begin{align*}
 \left| \int_{{\mathbb R}}y\left[G'(y) -G'(0)\right] d \mu(y) \right| 
\leq C_{\alpha,\beta,\eta} \left[A^{\beta-\alpha} \|G'\|_{\beta,\mbox{H\"ol}} + A^{2-\alpha} \|G"\|_{\infty}\right]
 \end{align*}
 Thus this achieves the proof of Proposition \ref{pro-maj-gen}  under Lemma \ref{lem-maj-mu-term}.
\end{proof}
 
\begin{lemma}\label{lem-maj-mu-term}
Under the assumptions of Proposition \ref{pro-maj-gen}, there exists a constant $C_{\alpha,\beta,\eta} $ such that for all $G \in {\mathcal F}^{\alpha,\beta}$ and $A>1$
\begin{align*}
\left| \int_{{\mathbb R}}y\left[G'(y) -G'(0]\right] d \mu(y) \right| 
\leq C_{\alpha,\beta,\eta} \left[A^{\beta-\alpha} \|G'\|_{\beta-1,\mbox{H\"ol}} + A^{2-\alpha} \|G"\|_{\infty}\right].
\end{align*}
\end{lemma}
\begin{proof}[Proof of Lemma \ref{lem-maj-mu-term}~:]
Let $A>1$
We split the integral in fourth terms :
\begin{align*}
\int_{{\mathbb R}}y\left[G'(y) -G'(0]\right] d \mu(y)&=\int_{\{|y| \leq 1\}}y\left[G'(y) -G'(0]\right] d \mu(y)\\
&+ \int_{\{|y| \geq A\}} y\left[G'(y) -G'(0)\right] d \mu(y)\\
&+ \int_1^A y\left[G'(y) -G'(0)\right] d \mu(y)+  \int_{-^A}^{-1} y\left[G'(y) -G'(0)\right] d \mu(y).
\end{align*}

The first one is bounded using Taylor expansion  by 
\begin{align*}
\left|\int_{\{|y| \leq 1\}}y\left[G'(y) -G'(0)\right] d \mu(y)\right| \leq \|G"\|_{\infty} |\mu({\mathbb R})|\leq 2 \|G"\|_{\infty}. 
\end{align*}

Using the fact that $G'$ is $\beta-1$ H\"older continuous and that $|\mu| \leq  (1+\|\varepsilon \|_{\infty}) {\mathbb P}_Z$ (an $\alpha$ Pareto distribution) the second term is bounded by
\begin{align*}
\left|\int_{\{|y| \geq A\}} y\left[G'(y) -G'(0)\right] \dif\mu(y)\right|
&\leq
\|G'\|_{\beta-1,\mbox{H\"ol}}(1+\|\varepsilon \|_{\infty}) \int_{\{|y| \geq A\}} |y|^{\beta}\dif {\mathbb P}_Z(y)\\
&\leq
\|G'\|_{\beta-1,\mbox{H\"ol}}(1+\|\varepsilon \|_{\infty}) \frac{ \alpha}{\beta- \alpha}\cdotp
 \end{align*}

For the third one, we perform an integration by part
\begin{align*}
\int_1^A y\left[G'(y) -G'(0)\right] \dif \mu(y)&=
 \left[ - y\left[G'(y) -G'(0)\right] \mu( ]y,+\infty[) \right]_1^A \\&
 + \int_1^A \left( \left[G'(y) -G'(0)\right] + y G"(y) \right) \mu( ]y,\infty[) \dif y.
\end{align*}
Thus  since 
$$ \left| \mu( ]y,+\infty[)\right|\leq  \frac{(1+ \|\varepsilon \|_{\infty})}{2}\frac{1}{y^{\alpha}
}$$
it is bounded by
\begin{align*}
\left| \int_1^A y\left[G'(y) -G'(0)\right] d \mu(y)\right|\leq 2\left[1+\frac{\alpha}{2-\alpha }\right] \|G"\|_{\infty}(1+ \|\varepsilon\|_{\infty} ) \left[ 1+ A^{2-\alpha}\right].
\end{align*}
The integral on $]-A,-1]$ term is bounded in the same spirit.
This achieves the proof of Lemma \ref{lem-maj-mu-term}, and thus the proof of Proposition \ref{pro-maj-gen} is complete.
\end{proof}

 Let us now return to the general case. Recall that the goal is to control
 $$
 \esp{F(\textbf{Y})} = \esp{\phi_p\left(\frac{Y_1 + \dots + Y_n}{n^{1/\alpha}}\right)},
 $$ 
where $Y_i$'s are in the normal domain of attraction.
 Hence, we need to control:
$$
E=\sum_{i=1}^n\esp{ -\frac{1}{\alpha}\partial_{y_i} G(\textbf{Y}) Y_i + \eta\int_{\R} \Big(G(\textbf{Y} + u \textbf{e}_i) - G(\textbf{Y}) - \nabla G(\textbf{Y}) \cdot u \Big)\frac{1}{|u|^{\alpha +1}}\dif u}.
$$

Keeping the notations as in the previous paragraph, we can reduce this calculation to the one dimensional case by conditioning. Let $G=P_tF$ and for $i=1,\dots,n$:
$${G}_{\backslash Y_i}(y) 
= G\Big(Y_1,\dots, Y_{i-1},y,Y_{i+1},\dots, Y_n\Big).$$
We need to control:
$$
E= \sum_{i=1}^n\esp{-\frac{1}{\alpha} {G}_{\backslash Y_i}'(Y_i) Y_i + \eta\int_{\R} \Big({G}_{\backslash Y_i}(Y_i + u) - {G}_{\backslash Y_i}(Y_i) - {G}_{\backslash Y_i}'(Y_i) \cdot u \Big)\frac{1}{|u|^{\alpha +1}}\dif u}.
$$

Now, we can plug the one-dimensional result from Proposition \ref{pro-maj-gen}:

\begin{align*}
E \leq c_{\alpha,\beta,\eta} \sum_{i=1}^n \esp{ \|G"_{\backslash Y_i}\|_{\infty}
+ (1-\eta)(1+\|\varepsilon\|_{\infty})\left\{A^{\beta-\alpha} \|G'_{\backslash Y_i}\|_{\beta-1,\mbox{H\"ol}} + A^{2-\alpha} \|G"_{\backslash Y_i}\|_{\infty}\right\} 
}.
\end{align*}

Finally, we recall that $\Vert G'' \Vert_\infty = \frac{1}{n^{2/\alpha}} e^{-2t/\alpha} \Vert \phi_p''\Vert_\infty $ and 
$$
\|G'\|_{\beta-1,\mbox{H\"ol}} \le C e^{-t/\alpha} \frac{1}{n^{\beta/\alpha}}\Vert \phi_p\Vert_\infty.
$$
This power in $n$ prompts us to chose $A = n^{\delta}$, where $\delta < \frac{1}{\alpha} $, so that 
$$
n \times A^{\beta-\alpha} \|G'_{\backslash Y_i}\|_{\beta-1,\mbox{H\"ol}} \underset{n \rightarrow +\infty} {\longrightarrow}0.
$$
A similar calculation leads to the same restriction over $\delta$ for the contribution of $A^{2-\alpha}\Vert G''_{\backslash Y_i} \Vert_{\infty}$.
This choice of $A$ lets us finally control the difference by something bounded in $n$. 
As for the integral in time, the extra $e^{-t/\alpha}$ factor allows us to integrate from 0 to $+\infty$,  thus the moment condition holds.

 \bibliography{MyLibrary.bib}

\end{document}